\documentclass[12pt, reqno]{amsart}

\usepackage{amssymb}  % AMS fonts
\usepackage{fullpage}
\usepackage{array}
\usepackage{enumitem}
\usepackage{url}
\usepackage{verbatim}

\theoremstyle{plain}
\newtheorem{theorem}{Theorem}[section]
\newtheorem{lemma}[theorem]{Lemma}

\newtheorem{proposition}[theorem]{Proposition}

\newtheorem{problem}[theorem]{Problem}

\theoremstyle{remark}
\newtheorem{example}[theorem]{Example}

\newtheorem*{acknowledgment}{Acknowledgment}

\numberwithin{equation}{section}

\newcommand{\seclabel}[1]{\label{sec:#1}}   % section
\newcommand{\thmlabel}[1]{\label{thm:#1}}   % theorem
\newcommand{\lemlabel}[1]{\label{lem:#1}}   % lemma
\newcommand{\exmlabel}[1]{\label{exm:#1}}   % example
\newcommand{\eqnlabel}[1]{\label{eqn:#1}}   % equation

\newcommand{\secref}[1]{\ref{sec:#1}}   % section
\newcommand{\thmref}[1]{\ref{thm:#1}}   % theorem
\newcommand{\lemref}[1]{\ref{lem:#1}}   % lemma
\newcommand{\exmref}[1]{\ref{exm:#1}}   % example
\newcommand{\eqnref}[1]{\eqref{eqn:#1}} % parenthesized eqn ref

\newcommand{\by}[1]{\overset{\eqnref{#1}}=}  % typesets reference to equation #1 above equality sign
\newcommand{\byx}[1]{\overset{#1}=}          % typesets reference to #1 above equality sign

\newcommand{\setof}[2]{\{#1\,|\,#2\}}   % set-builder notation

\title{Independent Axiom Systems for Nearlattices}

\author{Jo\~{a}o Ara\'{u}jo$^*$}
\author{Michael Kinyon}

\address[Ara\'{u}jo]
{Universidade Aberta
and
Centro de \'{A}lgebra \\
Universidade de Lisboa \\
1649-003 Lisboa \\ Portugal}
\email{\url{jaraujo@ptmat.fc.ul.pt}}

\address[Kinyon]{Department of Mathematics \\
University of Denver \\ 2360 S Gaylord St \\ Denver, Colorado 80208 USA}
\email{\url{mkinyon@math.du.edu}}

\thanks{${}^*$Partially supported by FCT and FEDER,
Project POCTI-ISFL-1-143 of Centro de Algebra da Universidade de Lisboa,
and by FCT and PIDDAC through the project PTDC/MAT/69514/2006.}


\begin{document}

\begin{abstract}
A \emph{nearlattice} is a join semilattice such that every principal
filter is a lattice with respect to the induced order. Hickman and later
Chajda \emph{et al} independently showed that nearlattices can be treated
as varieties of algebras with a ternary operation satisfying certain axioms.
Our main result is that the variety of nearlattices is $2$-based, and we
exhibit an explicit system of two independent identities. We also
show that the original axiom systems of Hickman and of Chajda \emph{et al}
are respectively dependent.
\end{abstract}

\maketitle

\section{Introduction}
\seclabel{intro}

A \emph{nearlattice} $(L,\lor,\{\land_a\}_{a\in L})$ is a join semilattice
$(L,\lor)$ such that for each
$a\in L$, the principal filter $(a] = \setof{x\in L}{a\leq x}$ is a lattice
with respect to the induced order. Each $x,y\in (a]$ has a meet, denoted
by $x\land_a y$. Define $m : L\times L\times L\to L$ by
\begin{equation}
\eqnlabel{ternary-m}
m(x,y,z) = (x\lor z)\land_z (y\lor z)\,,
\end{equation}
for $x,y,z\in L$. Sometimes the literature on nearlattices considers instead
the dual object consisting of a meet semilattice such that each principal
ideal is a lattice.

Hickman \cite{H} and, independently,
Chajda and Hala\v{s} \cite{CH} and Chajda and Kola\v{r}\'{i}k \cite{CK} characterized
nearlattices in terms of the ternary operation $m$. Part (1) of the
following proposition gives Hickman's axioms and part (2) gives the
axioms of Chajda \emph{et al} as found in, for instance, \cite{CHK}, \S{2.6}.

\begin{proposition}
\begin{enumerate}
\item Let $(L,\lor,\{\land_a\}_{a\in L})$ be a nearlattice. Then $m: L\times L\times L\to L$
defined by \eqnref{ternary-m} satisfies the following identities:
\begin{align*}
m(x,x,x) &= x  \tag{H1} \\
m(x,x,y) &= m(y,y,x) \tag{H2} \\
m(m(x,x,y),m(x,x,y),z) &= m(x,x,m(y,y,z)) \tag{H3} \\
m(x,y,z) &= m(y,x,z)  \tag{H4} \\
m(m(x,y,z),m(x,y,z),m(x,x,z)) &= m(x,x,z)  \tag{H5} \\
m(m(x,x,y),z,y) &= m(x,z,y) \tag{H6} \\
m(x,m(x,x,y),z) &= m(x,x,z) \tag{H7} \\
m(m(x,m(y,z,u),u),m(x,m(y,z,u),u),m(x,z,u)) &= m(x,z,u)\,. \tag{H8}
\end{align*}
Conversely, let $(L,m)$ be a ternary algebra satisfying the identities
(P1)--(P8). Define $x\lor y = m(x,x,y)$. Then $(L,\lor)$ is a semilattice.
For each $a\in L$ and all $x,y\in (a]$, define $x \land_a y = m(x,y,a)$.
Then $(L,\lor,\{\land_a\}_{a\in L})$ is a nearlattice.
\item Let $(L,\lor,\{\land_a\}_{a\in L})$ be a nearlattice. Then $m: L\times L\times L\to L$
defined by \eqnref{ternary-m} satisfies the following identities:
\begin{align*}
m(x,y,x) &= x \tag{P1} \\
m(x,x,y) &= m(y,y,x) \tag{P2} \\
m(m(x,x,y),m(x,x,y),z) &= m(x,x,m(y,y,z)) \tag{P3} \\
m(x,y,z) &= m(y,x,z) \tag{P4} \\
m(m(x,y,z),w,z) &= m(x,m(y,w,z),z) \tag{P5} \\
m(x,m(y,y,x),z) &= m(x,x,z) \tag{P6} \\
m(x,x,m(x,y,z)) &= m(x,x,z) \tag{P7} \\
m(m(x,x,z),m(y,y,z),z) &= m(x,y,z)\,. \tag{P8}
\end{align*}
Conversely, let $(L,m)$ be a ternary algebra satisfying the identities
(P1)--(P8). Define $x\lor y = m(x,x,y)$. Then $(L,\lor)$ is a semilattice.
For each $a\in L$ and all $x,y\in (a]$, define $x \land_a y = m(x,y,a)$.
Then $(L,\lor,\{\land_a\}_{a\in L})$ is a nearlattice.
\end{enumerate}
\end{proposition}

(In fact, Hickman worked with the dual notion of nearlattice, and used a different
convention for ordering the variables. Our $m(x,y,z)$ is Hickman's $j(x,z,y)$.)

Thus nearlattices can be treated as varieties of algebras, and from now on, we
will refer to the ternary structures $(L,m)$ themselves as nearlattices.
The main result of this paper is that the variety of nearlattices $(L,m)$ is $2$-based.

\begin{theorem}
\thmlabel{2base}
The following is a basis of identities for nearlattices.
\begin{align*}
m(x,y,x) &= x \tag{N1} \\
m(m(x,y,z),m(y,m(u,x,z),z),w) &= m(w,w,m(y,m(x,u,z),z))\,. \tag{N2}
\end{align*}
\end{theorem}

Left open in both the investigations of Hickman and of Chajda \emph{et al}
was the independence of their respective axiom systems.
In fact, three of Hickman's axioms are dependent upon the others,
and one of the axioms of Chajda \emph{et al} is dependent upon the rest.

\begin{theorem}
\thmlabel{hickman}
The system
$\mathcal{H} = \{$\emph{(H1)}, \emph{(H2)}, \emph{(H4)}, \emph{(H7)}, \emph{(H8)}$\}$
is a basis of independent identities for the variety of nearlattices. In particular,
the identities $\mathcal{H}$ imply \emph{(H3)}, \emph{(H5)} and \emph{(H6)}.
\end{theorem}

\begin{theorem}
\thmlabel{CH}
The system
$\mathcal{C} = \{$\emph{(P1)}, \emph{(P2)}, \emph{(P4)}--\emph{(P8)}$\}$
is a basis of independent identities for the variety of nearlattices. In particular,
the identities $\mathcal{C}$ imply \emph{(P3)}.
\end{theorem}

In \S\secref{hickman}, \S\secref{CH} and \S\secref{main}, we prove Theorems
\thmref{hickman}, \thmref{CH} and \thmref{2base}, respectively. We conclude in
\S\secref{problems} with some open problems.

\section{Hickman's Axioms}
\seclabel{hickman}

In this section we prove Theorem \thmref{hickman}. Assume now
that $(A,m)$ is an algebra satisfying (H1), (H2), (H4), (H7) and (H8).

The identity (H8) is taken directly from Hickman's paper \cite{H} after appropriate
change of notation. There is, however, a somewhat more useful equivalent form.

\begin{lemma}
\lemlabel{h8p}
In the presence of (H2) and (H4), the identity (H8) is equivalent to
\[
m(m(x,y,z),m(x,y,z),m(x,m(y,u,z),z)) = m(x,y,z)\,. \tag{H8'}
\]
\end{lemma}

\begin{proof}
Starting with the left of (H8'), we have
\begin{align*}
& m(m(x,y,z),m(x,y,z),m(x,m(y,u,z),z))\\
 &\byx{(H4)} m(m(x,y,z),m(x,y,z),m(x,m(u,y,z),z)) \\
&\byx{(H2)} m(m(x,m(u,y,z),z),m(x,m(u,y,z),z),m(x,y,z))\,,
\end{align*}
which is the left side of (H8). Since the right sides of (H8) and (H8')
coincide, this proves the desired equivalence.
\end{proof}

Next we need a pair of useful identities.

\begin{lemma}
For all $x,y$,
\begin{align}
m(x,y,x) &= x \eqnlabel{xyxx} \\
m(x,y,y) &= y \eqnlabel{xyyy}
\end{align}
\end{lemma}

\begin{proof}

First, we have
\begin{align*}
m(x,x,m(x,y,x)) &\byx{(H2)} m(m(x,y,x),m(x,y,x),\underbrace{x}) \\
&\byx{(H1)} m(m(x,y,x),m(x,y,x),m(x,x,x)) \\
&\byx{(H7)} m(m(x,y,x),m(x,y,x),m(x,m(x,x,y),x)) \\
&\byx{(H2)} m(m(x,y,x),m(x,y,x),m(x,m(y,y,x),x)) \\
&\byx{(H8')} m(x,y,x)\,,
\end{align*}
that is,
\begin{equation}
\eqnlabel{xyxtmp1}
m(x,x,m(x,y,x)) = m(x,y,x)\,.
\end{equation}

Next,
\begin{align*}
m(x,m(x,y,x),x) &\by{xyxtmp1} m(\underbrace{x},\underbrace{x},m(x,m(x,y,x),x)) \\
&\byx{(H1)} m(m(x,x,x),m(x,x,x),m(x,m(x,y,x),x)) \\
&\byx{(H8')} m(x,x,x) \\
&\byx{(H1)} x\,,
\end{align*}
that is,
\begin{equation}
\eqnlabel{xyxtmp2}
m(x,m(x,y,x),x) = x\,.
\end{equation}

In (H8'), take $z = y$ and $x = m(y,v,y)$. The left side of (H8') becomes
\begin{align*}
&m(\underbrace{m(m(y,v,y),y,y)},\underbrace{m(m(y,v,y),y,y)},m(m(y,v,y),m(y,u,y),y)) \\
&\byx{(H4)}
m(\underbrace{m(y,m(y,v,y),y)},\underbrace{m(y,m(y,v,y),y)},m(m(y,v,y),m(y,u,y),y)) \\
&\by{xyxtmp2} m(y,y,m(m(y,v,y),m(y,u,y),y))\,.
\end{align*}
The right side of (H8') becomes
\[
m(m(y,v,y),y,y) \byx{(H4)} m(y,m(y,v,y),y) \by{xyxtmp2} y\,.
\]
Putting this together, replace $y$ with $x$, $v$ with $y$ and $u$ with $z$ to get
\begin{equation}
\eqnlabel{xyxtmp3}
m(x,x,m(m(x,y,x),m(x,z,x),x)) = x\,.
\end{equation}

Now we can verify \eqnref{xyxx}:
\begin{align*}
m(x,y,x) &\by{xyxtmp1} m(x,x,\underbrace{m(x,y,x)}) \\
&\by{xyxtmp1} m(x,x,\underbrace{m(x,x,m(x,y,x))}) \\
&\byx{(H2)} m(x,x,m(m(x,y,x),m(x,y,x),x)) \\
&\by{xyxtmp3} x\,.
\end{align*}

Finally, \eqnref{xyyy} follows easily: $m(x,y,y) \byx{(H4)} m(y,x,y) \by{xyxx} y$.
\end{proof}

Next we work toward (H5). Key to this are the following identities.

\begin{lemma}
For all $x,y,z$,
\begin{align}
m(m(x,y,z),m(x,y,z),z) &= m(x,y,z) \eqnlabel{h5-9} \\
m(x,m(x,y,z),z) &= m(x,y,z) \eqnlabel{h5-12}
\end{align}
\end{lemma}

\begin{proof}
For \eqnref{h5-9}, we have
\begin{align*}
m(m(x,y,z),m(x,y,z),\underbrace{z}) &\by{xyyy} m(m(x,y,z),m(x,y,z),m(x,\underbrace{z},z)) \\
&\by{xyyy} m(m(x,y,z),m(x,y,z),m(x,m(y,z,z),z)) \\
&\byx{(H8')} m(x,y,z)\,.
\end{align*}

Then we compute
\begin{align*}
m(x,m(x,y,z),z) &\byx{(H4)} m(m(x,y,z),x,z) \\
&\byx{(H8')} m(m(m(x,y,z),x,z),m(m(x,y,z),x,z),m(m(x,y,z),m(x,y,z),z)) \\
&\by{h5-9} m(m(m(x,y,z),x,z),m(m(x,y,z),x,z),m(x,y,z)) \\
&\byx{(H2)} m(m(x,y,z),m(x,y,z),\underbrace{m(m(x,y,z),x,z)}) \\
&\byx{(H4)} m(m(x,y,z),m(x,y,z),m(x,\underbrace{m(x,y,z)},z)) \\
&\byx{(H4)} m(m(x,y,z),m(x,y,z),m(x,m(y,x,z),z)) \\
&\byx{(H8')} m(x,y,z)\,.
\end{align*}
This establishes \eqnref{h5-12}.
\end{proof}

\begin{lemma}
\lemlabel{H5}
If $(A,m)$ is an algebra satisfying (H1), (H2), (H4), (H7) and (H8),
then (H5) holds.
\end{lemma}

\begin{proof}
We compute
\begin{align*}
m(m(x,y,z),m(x,y,z),m(x,x,z)) &\byx{(H2)} m(m(x,x,z),m(x,x,z),\underbrace{m(x,y,z)}) \\
&\by{h5-12} m(m(x,x,z),m(x,x,z),m(x,m(x,y,z),z)) \\
&\byx{(H8')} m(x,x,z)\,,
\end{align*}
which proves the desired result.
\end{proof}

We continue with the assumptions of this section that we have an
algebra $(A,m)$ satisfying (H1), (H2), (H4), (H7) and (H8).
By Lemma \lemref{H5}, we may now freely use (H5).
Our next goal is to establish (H6).

\begin{lemma}
For all $x,y,z$,
\begin{equation}
\eqnlabel{h6-17}
m(m(x,x,y),m(z,x,y),y) = m(z,x,y)\,.
\end{equation}
\end{lemma}

\begin{proof}
We compute
\begin{align*}
m(m(x,x,y),m(z,x,y),y) &\byx{(H4)} m(\underbrace{m(z,x,y)},m(x,x,y),y) \\
&\byx{(H4)} m(m(x,z,y),\underbrace{m(x,x,y)},y) \\
&\byx{(H5)} m(m(x,z,y),m(m(x,z,y),m(x,z,y),m(x,x,y)),y) \\
&\byx{(H7)} m(m(x,z,y),m(x,z,y),y) \\
&\byx{(H2)} m(y,y,m(x,z,y)) \\
&\by{h5-12} m(x,z,y) \\
&\byx{(H4)} m(z,x,y)\,,
\end{align*}
which establishes the desired result.
\end{proof}

\begin{lemma}
\lemlabel{H6}
If $(A,m)$ is an algebra satisfying (H1), (H2), (H4), (H7) and (H8),
then (H6) holds.
\end{lemma}

\begin{proof}
We compute
\begin{align*}
m(m(x,x,y),z,y) &\byx{(H8')}
m(m(m(x,x,y),z,y),m(m(x,x,y),z,y),\underbrace{m(m(x,x,y),m(z,x,y),y)}) \\
&\by{h6-17} m(m(m(x,x,y),z,y),m(m(x,x,y),z,y),m(z,x,y)) \\
&\byx{(H2)} m(m(z,x,y),m(z,x,y),\underbrace{m(m(x,x,y),z,y)}) \\
&\byx{(H4)} m(m(z,x,y),m(z,x,y),m(z,m(x,x,y),y)) \\
&\byx{(H8')} m(z,x,y) \\
&\byx{(H4)} m(x,z,y)\,,
\end{align*}
which establishes (H6).
\end{proof}

The next goal is to verify (H3). We may now use
(H5) and (H6) freely in calculations.

\begin{lemma}
For all $x,y,z,u$,
\begin{align}
m(x,m(y,y,z),y) &= m(x,z,y) \eqnlabel{hick13} \\
m(x,x,m(x,y,z)) &= m(z,z,x) \eqnlabel{hick19} \\
m(x,m(y,z,u),z) &= m(x,u,z) \eqnlabel{hick21} \\
m(x,m(y,y,z),m(y,y,u)) &= m(x,z,m(y,y,u)) \eqnlabel{h3-22}
\end{align}
\end{lemma}

\begin{proof}
For \eqnref{hick13}, we have
\begin{align*}
m(x,m(y,y,z),y) &\byx{(H2)} m(x,m(z,z,y),y) \\
&\byx{(H4)} m(m(z,z,y),x,y) \\
&\byx{(H6)} m(z,x,y)\\
&\byx{(H4)} m(x,z,y)\,.
\end{align*}

For \eqnref{hick19}, we compute
\begin{align*}
m(x,x,m(x,y,z)) &\byx{(H2)} m(m(x,y,z),m(x,y,z),x) \\
&\byx{(H7)} m(m(x,y,z),\underbrace{m(m(x,y,z),m(x,y,z),m(x,x,z))},x) \\
&\byx{(H5)} m(m(x,y,z),m(x,x,z),x) \\
&\by{hick13} m(m(x,y,z),z,x) \\
&\byx{(H4)} m(z,\underbrace{m(x,y,z)},x) \\
&\by{h5-9} m(z,\underbrace{m(m(x,y,z),m(x,y,z),z)},x) \\
&\byx{(H2)} m(z,m(z,z,m(x,y,z)),x) \\
&\byx{(H7)} m(z,z,x)\,.
\end{align*}

For \eqnref{hick21}, we have
\begin{align*}
m(x,\underbrace{m(y,z,u)},z) &\byx{(H4)} m(x,\underbrace{m(z,y,u)},z) \\
&\by{hick13} m(x,\underbrace{m(z,z,m(z,y,u))},z) \\
&\by{hick19} m(x,m(u,u,z),z) \\
&\byx{(H4)} m(m(u,u,z),x,z) \\
&\byx{(H6)} m(u,x,z)\\
&\byx{(H4)} m(x,u,z)\,.
\end{align*}

Finally, for \eqnref{h3-22}, we have
\begin{align*}
m(x,\underbrace{m(y,y,z)},m(y,y,u)) &\byx{(H7)} m(x,m(y,m(y,y,u),z),m(y,y,u)) \\
&\by{hick21} m(x,z,m(y,y,u))\,.
\end{align*}
This completes the proof.
\end{proof}

\begin{lemma}
\lemlabel{H3}
If $(A,m)$ is an algebra satisfying (H1), (H2), (H4), (H7) and (H8),
then (H3) holds.
\end{lemma}

\begin{proof}
We compute
\begin{align*}
m(x,x,m(y,y,z)) &\by{h3-22} m(x,m(y,y,x),m(y,y,z)) \\
&\byx{(H4)} m(m(y,y,x),x,m(y,y,z)) \\
&\by{h3-22} m(m(y,y,x),m(y,y,x),m(y,y,z)) \\
&\byx{(H2)} m(m(y,y,z),m(y,y,z),m(y,y,x)) \\
&\by{h3-22} m(m(y,y,z),z,m(y,y,x)) \\
&\byx{(H4)} m(z,m(y,y,z),m(y,y,x)) \\
&\by{h3-22} m(z,z,\underbrace{m(y,y,x)})\\
&\byx{(H2)} m(z,z,m(x,x,y)) \\
&\byx{(H2)} m(m(x,x,y),m(x,x,y),z)\,.
\end{align*}
This establishes (H3) as claimed.
\end{proof}

Next, we check the independence of the axioms (H1), (H2), (H4), (H7) and (H8).
We found these models using \textsc{Mace4} \cite{McCune}. We simply state
the models and leave the verification that each model has its claimed properties
to the reader.

\begin{example}
\exmlabel{notH8}
An example of an algebra $(S,m)$ satisfying (H1), (H2), (H4) and
(H7), but not (H8) is given by $S =\{0,1,2\}$ with $m:S\times S\times S\to S$ defined
by the following tables.

\begin{table}[htb]
\centering
\begin{tabular}{r|ccc}
$m(0,\cdot,\cdot)$ & 0 & 1 & 2 \\
\hline
0 &  0 & 2 & 2 \\
1 &  0 & 0 & 0 \\
2 &  0 & 2 & 2
\end{tabular}
\quad
\begin{tabular}{r|ccc}
$m(1,\cdot,\cdot)$ & 0 & 1 & 2 \\
\hline
0 &  0 & 0 & 0 \\
1 &  2 & 1 & 2 \\
2 &  2 & 1 & 2
\end{tabular}
\quad
\begin{tabular}{r|ccc}
$m(2,\cdot,\cdot)$ & 0 & 1 & 2 \\
\hline
0 &  0 & 2 & 2 \\
1 &  2 & 1 & 2 \\
2 &  2 & 2 & 2
\end{tabular}
\end{table}
\end{example}

\begin{example}
\exmlabel{notH7}
An example of an algebra $(S,m)$ satisfying (H1), (H2), (H4) and
(H8), but not (H7) is given by $S =\{0,1,2\}$ with $m:S\times S\times S\to S$ defined
by the following tables.

\begin{table}[htb]
\centering
\begin{tabular}{r|ccc}
$m(0,\cdot,\cdot)$ & 0 & 1 & 2 \\
\hline
0 &  0 & 1 & 0 \\
1 &  0 & 1 & 2 \\
2 &  0 & 1 & 2
\end{tabular}
\quad
\begin{tabular}{r|ccc}
$m(1,\cdot,\cdot)$ & 0 & 1 & 2 \\
\hline
0 &  0 & 1 & 2 \\
1 &  1 & 1 & 1 \\
2 &  0 & 1 & 2
\end{tabular}
\quad
\begin{tabular}{r|ccc}
$m(2,\cdot,\cdot)$ & 0 & 1 & 2 \\
\hline
0 &  0 & 1 & 2 \\
1 &  0 & 1 & 2 \\
2 &  0 & 1 & 2
\end{tabular}
\end{table}
\end{example}

\begin{example}
\exmlabel{notH4}
An example of an algebra $(S,m)$ satisfying (H1), (H2), (H7) and
(H8), but not (H4) is given by $S =\{0,1\}$ with $m:S\times S\times S\to S$ defined
by the following tables.

\begin{table}[htb]
\centering
\begin{tabular}{r|cc}
$m(0,\cdot,\cdot)$ & 0 & 1 \\
\hline
0 &  0 & 1  \\
1 &  0 & 1
\end{tabular}
\quad
\begin{tabular}{r|cc}
$m(1,\cdot,\cdot)$ & 0 & 1 \\
\hline
0 &  1 & 1 \\
1 &  1 & 1
\end{tabular}
\end{table}
\end{example}

\begin{example}
\exmlabel{notH2}
An example of an algebra $(S,m)$ satisfying (H1), (H4), (H7) and
(H8), but not (H2) is given by $S =\{0,1\}$ with $m:S\times S\times S\to S$ defined
by the following tables.

\begin{table}[htb]
\centering
\begin{tabular}{r|cc}
$m(0,\cdot,\cdot)$ & 0 & 1 \\
\hline
0 &  0 & 1  \\
1 &  0 & 1
\end{tabular}
\quad
\begin{tabular}{r|cc}
$m(1,\cdot,\cdot)$ & 0 & 1 \\
\hline
0 &  0 & 1 \\
1 &  0 & 1
\end{tabular}
\end{table}
\end{example}

\begin{example}
\exmlabel{notH1}
An example of an algebra $(S,m)$ satisfying (H2), (H4), (H7) and
(H8), but not (H1) is given by $S =\{0,1\}$ with $m:S\times S\times S\to S$ defined
by the following tables.

\begin{table}[htb]
\centering
\begin{tabular}{r|cc}
$m(0,\cdot,\cdot)$ & 0 & 1 \\
\hline
0 &  1 & 1  \\
1 &  1 & 1
\end{tabular}
\quad
\begin{tabular}{r|cc}
$m(1,\cdot,\cdot)$ & 0 & 1 \\
\hline
0 &  1 & 1 \\
1 &  1 & 1
\end{tabular}
\end{table}
\end{example}

Putting together Lemmas \lemref{H5}, \lemref{H6} and \lemref{H3}, along
with Examples \exmref{notH8}, \exmref{notH7}, \exmref{notH4}, \exmref{notH2}
and \exmref{notH1}, we have completed the proof of Theorem \thmref{hickman}.

\section{Axioms of Chajda \emph{et al}}
\seclabel{CH}

In this section we prove Theorem \thmref{CH}. Assume now
that we have an algebra $(A,m)$ satisfying (P1), (P2) and (P4)--(P8).

\begin{lemma}
For all $x,y,z$,
\begin{equation}
\eqnlabel{p11}
m(m(x,y,z),m(x,x,z),y) = m(y,y,z)\,.
\end{equation}
\end{lemma}

\begin{proof}
We compute
\begin{align*}
m(m(x,y,z),\underbrace{m(x,x,z)},y) &\byx{(P7)} m(m(x,y,z),m(x,x,m(x,y,z)),y) \\
&\byx{(P6)} m(m(x,y,z),m(x,y,z),y) \\
&\byx{(P2)} m(y,y,\underbrace{m(x,y,z)}) \\
&\byx{(P4)} m(y,y,m(y,x,z)) \\
&\byx{(P7)} m(y,y,z)\,,
\end{align*}
which establishes the claim.
\end{proof}

\begin{lemma}
\lemlabel{P3}
Let $(A,m)$ be an algebra satisfying (P1), (P2) and (P4)--(P8).
Then $(A,m)$ satisfies (P3).
\end{lemma}

\begin{proof}
First, we have
\begin{align*}
m(x,x,m(y,y,z)) &\byx{(P2)} m(m(y,y,z),m(y,y,z),x) \\
&\byx{(P7)} m(m(y,y,z),m(y,y,z),\underbrace{m(m(y,y,z),m(y,x,z),x)}) \\
&\byx{(P4)} m(m(y,y,z),m(y,y,z),\underbrace{m(m(y,x,z),m(y,y,z),x)}) \\
&\by{p11} m(m(y,y,z),m(y,y,z),m(x,x,y))\,,
\end{align*}
that is,
\begin{equation}
\eqnlabel{halfofit}
m(x,x,m(y,y,z)) = m(m(y,y,z),m(y,y,z),m(x,x,y))\,.
\end{equation}

Now
\begin{align*}
m(x,x,m(y,y,z)) &\by{halfofit} m(m(y,y,z),m(y,y,z),m(x,x,y)) \\
&\byx{(P2)} m(\underbrace{m(x,x,y)},\underbrace{m(x,x,y)},\underbrace{m(y,y,z)}) \\
&\byx{(P2)} m(m(y,y,x),m(y,y,x),m(z,z,y)) \\
&\by{halfofit} m(z,z,\underbrace{m(y,y,x)}) \\
&\byx{(P2)} m(z,z,m(x,x,y)) \\
&\byx{(P2)} m(m(x,x,y),m(x,x,y),z)\,,
\end{align*}
which establishes (P3).
\end{proof}

Note that we used only (P2), (P4), (P6) and (P7) in the proof of (P3).

Next we consider the independence of the axioms (P1), (P2) and (P4)--(P8).
As in the previous section, we simply give the models and leave the verification
that each model has its claimed properties to the reader.

\begin{example}
\exmlabel{notP8}
An example of an algebra $(S,m)$ satisfying (P1), (P2), (P4)--(P7), but not (P8) is given
by $S =\{0,1,2,3\}$ with $m:S\times S\times S\to S$ defined
by the following tables.

\begin{table}[htb]
\centering
\begin{tabular}{lr}
\begin{tabular}{r|cccc}
$m(0,\cdot,\cdot)$ & 0 & 1 & 2 & 3 \\
\hline
0 &  0 & 1 & 1 & 1 \\
1 &  0 & 1 & 1 & 1 \\
2 &  0 & 3 & 2 & 1 \\
3 &  0 & 1 & 1 & 3
\end{tabular}
&
\begin{tabular}{r|cccc}
$m(1,\cdot,\cdot)$ & 0 & 1 & 2 & 3 \\
\hline
0 &  0 & 1 & 1 & 1 \\
1 &  1 & 1 & 1 & 1 \\
2 &  1 & 1 & 2 & 1 \\
3 &  1 & 1 & 1 & 3
\end{tabular}
\\ \\
\begin{tabular}{r|cccc}
$m(2,\cdot,\cdot)$ & 0 & 1 & 2 & 3 \\
\hline
0 &  0 & 3 & 2 & 1 \\
1 &  1 & 1 & 2 & 1 \\
2 &  1 & 1 & 2 & 1 \\
3 &  1 & 1 & 2 & 3
\end{tabular}
&
\begin{tabular}{r|cccc}
$m(3,\cdot,\cdot)$ & 0 & 1 & 2 & 3 \\
\hline
0 &  0 & 1 & 1 & 3 \\
1 &  1 & 1 & 1 & 3 \\
2 &  1 & 1 & 2 & 3 \\
3 &  1 & 1 & 1 & 3
\end{tabular}
\end{tabular}
\end{table}
\end{example}

\begin{example}
\exmlabel{notP7}
An example of an algebra $(S,m)$ satisfying (P1), (P2), (P4)--(P6), (P8), but not (P7) is given
by $S =\{0,1,2,3\}$ with $m:S\times S\times S\to S$ defined
by the following tables.

\begin{table}[htb]
\centering
\begin{tabular}{lr}
\begin{tabular}{r|cccc}
$m(0,\cdot,\cdot)$ & 0 & 1 & 2 & 3 \\
\hline
0 &  0 & 3 & 0 & 3 \\
1 &  0 & 1 & 1 & 3 \\
2 &  0 & 1 & 2 & 3 \\
3 &  0 & 3 & 0 & 3
\end{tabular}
&
\begin{tabular}{r|cccc}
$m(1,\cdot,\cdot)$ & 0 & 1 & 2 & 3 \\
\hline
0 &  0 & 1 & 1 & 3 \\
1 &  3 & 1 & 1 & 3 \\
2 &  0 & 1 & 2 & 3 \\
3 &  3 & 1 & 1 & 3
\end{tabular}
\\ \\
\begin{tabular}{r|cccc}
$m(2,\cdot,\cdot)$ & 0 & 1 & 2 & 3 \\
\hline
0 &  0 & 1 & 2 & 3 \\
1 &  0 & 1 & 2 & 3 \\
2 &  0 & 1 & 2 & 3 \\
3 &  0 & 1 & 2 & 3
\end{tabular}
&
\begin{tabular}{r|cccc}
$m(3,\cdot,\cdot)$ & 0 & 1 & 2 & 3 \\
\hline
0 &  0 & 3 & 0 & 3 \\
1 &  3 & 1 & 1 & 3 \\
2 &  0 & 1 & 2 & 3 \\
3 &  3 & 3 & 3 & 3
\end{tabular}
\end{tabular}
\end{table}
\end{example}

\begin{example}
\exmlabel{notP6}
An example of an algebra $(S,m)$ satisfying (P1), (P2), (P4), (P5), (P7), (P8), but not (P6) is given
by $S =\{0,1,2\}$ with $m:S\times S\times S\to S$ defined
by the following tables.

\begin{table}[htb]
\centering
\begin{tabular}{r|ccc}
$m(0,\cdot,\cdot)$ & 0 & 1 & 2 \\
\hline
0 &  0 & 1 & 0 \\
1 &  0 & 1 & 2 \\
2 &  0 & 1 & 2
\end{tabular}
\quad
\begin{tabular}{r|ccc}
$m(1,\cdot,\cdot)$ & 0 & 1 & 2 \\
\hline
0 &  0 & 1 & 2 \\
1 &  1 & 1 & 1 \\
2 &  0 & 1 & 2
\end{tabular}
\quad
\begin{tabular}{r|ccc}
$m(2,\cdot,\cdot)$ & 0 & 1 & 2 \\
\hline
0 &  0 & 1 & 2 \\
1 &  0 & 1 & 2 \\
2 &  0 & 1 & 2
\end{tabular}
\end{table}
\end{example}

\begin{example}
\exmlabel{notP5}
An example of an algebra $(S,m)$ satisfying (P1), (P2), (P4), (P6)--(P8), but not (P5) is given
by $S =\{0,1,2,3,4\}$ with $m:S\times S\times S\to S$ defined
by the following tables.

\begin{table}[htb]
\centering
\begin{tabular}{r|ccccc}
$m(0,\cdot,\cdot)$ & 0 & 1 & 2 & 3 & 4 \\
\hline
0 &  0 & 0 & 3 & 3 & 0 \\
1 &  0 & 1 & 2 & 3 & 1 \\
2 &  0 & 4 & 2 & 3 & 1 \\
3 &  0 & 0 & 3 & 3 & 0 \\
4 &  0 & 1 & 2 & 3 & 4
\end{tabular}
\quad
\begin{tabular}{r|ccccc}
$m(1,\cdot,\cdot)$ & 0 & 1 & 2 & 3 & 4 \\
\hline
0 &  0 & 1 & 2 & 3 & 1 \\
1 &  0 & 1 & 2 & 3 & 1 \\
2 &  0 & 1 & 2 & 3 & 1 \\
3 &  0 & 1 & 2 & 3 & 1 \\
4 &  0 & 1 & 2 & 3 & 4
\end{tabular}
\quad
\begin{tabular}{r|ccccc}
$m(2,\cdot,\cdot)$ & 0 & 1 & 2 & 3 & 4 \\
\hline
0 &  0 & 4 & 2 & 3 & 1 \\
1 &  0 & 1 & 2 & 3 & 1 \\
2 &  3 & 2 & 2 & 3 & 2 \\
3 &  3 & 2 & 2 & 3 & 2 \\
4 &  0 & 1 & 2 & 3 & 4
\end{tabular}

\begin{tabular}{r|ccccc}
$m(3,\cdot,\cdot)$ & 0 & 1 & 2 & 3 & 4 \\
\hline
0 &  0 & 0 & 3 & 3 & 0 \\
1 &  0 & 1 & 2 & 3 & 1 \\
2 &  3 & 2 & 2 & 3 & 2 \\
3 &  3 & 3 & 3 & 3 & 3 \\
4 &  0 & 1 & 2 & 3 & 4
\end{tabular}
\quad
\begin{tabular}{r|ccccc}
$m(4,\cdot,\cdot)$ & 0 & 1 & 2 & 3 & 4 \\
\hline
0 &  0 & 1 & 2 & 3 & 4 \\
1 &  0 & 1 & 2 & 3 & 4 \\
2 &  0 & 1 & 2 & 3 & 4 \\
3 &  0 & 1 & 2 & 3 & 4 \\
4 &  0 & 1 & 2 & 3 & 4
\end{tabular}
\end{table}
\end{example}

\begin{example}
\exmlabel{notP4}
An example of an algebra $(S,m)$ satisfying (P1), (P2), (P5)--(P8), but not (P4) is given
by $S =\{0,1\}$ with $m:S\times S\times S\to S$ defined
by the following tables.

\begin{table}[htb]
\centering
\begin{tabular}{r|cc}
$m(0,\cdot,\cdot)$ & 0 & 1 \\
\hline
0 &  0 & 1  \\
1 &  0 & 1
\end{tabular}
\quad
\begin{tabular}{r|cc}
$m(1,\cdot,\cdot)$ & 0 & 1 \\
\hline
0 &  1 & 1 \\
1 &  1 & 1
\end{tabular}
\end{table}
\end{example}

\begin{example}
\exmlabel{notP2}
An example of an algebra $(S,m)$ satisfying (P1), (P4)--(P8), but not (P2) is given
by $S =\{0,1\}$ with $m:S\times S\times S\to S$ defined
by the following tables.

\begin{table}[htb]
\centering
\begin{tabular}{r|cc}
$m(0,\cdot,\cdot)$ & 0 & 1 \\
\hline
0 &  0 & 1  \\
1 &  0 & 1
\end{tabular}
\quad
\begin{tabular}{r|cc}
$m(1,\cdot,\cdot)$ & 0 & 1 \\
\hline
0 &  0 & 1 \\
1 &  0 & 1
\end{tabular}
\end{table}
\end{example}

\begin{example}
\exmlabel{notP1}
An example of an algebra $(S,m)$ satisfying (P2), (P4)--(P8), but not (P1) is given
by $S =\{0,1\}$ with $m:S\times S\times S\to S$ defined
by the following tables.

\begin{table}[htb]
\centering
\begin{tabular}{r|cc}
$m(0,\cdot,\cdot)$ & 0 & 1 \\
\hline
0 &  1 & 1  \\
1 &  1 & 1
\end{tabular}
\quad
\begin{tabular}{r|cc}
$m(1,\cdot,\cdot)$ & 0 & 1 \\
\hline
0 &  1 & 1 \\
1 &  1 & 1
\end{tabular}
\end{table}
\end{example}

Putting together Lemma \lemref{P3} with Examples \exmref{notP8}, \exmref{notP7},
\exmref{notP6}, \exmref{notP5}, \exmref{notP4}, \exmref{notP2}
and \exmref{notP1}, we have completed the proof of Theorem \thmref{CH}.

\section{A $2$-Base For Nearlattices}
\seclabel{main}

In this section we prove Theorem \thmref{2base}. We first prove the easy direction.

\begin{lemma}
\lemlabel{easy}
Every nearlattice satisfies (N1) and (N2).
\end{lemma}

\begin{proof}
We are, of course, free to use Hickman's axioms and the axioms of Chajda \emph{et al},
together with any derived properties from the previous two sections.
The identity (N1) is just (P1), so (N2) is the only identity requiring proof:
\begin{align*}
& m(m(x,y,z),m(y,\underbrace{m(u,x,z)},z),w) \\
&\byx{(H4)} m(\underbrace{m(x,y,z)},m(y,m(x,u,z),z),w) \\
&\byx{(H4)} m(m(y,x,z),m(y,m(x,u,z),z),w) \\
&\byx{(H4)} m(m(y,m(x,u,z),z),\underbrace{m(y,x,z)},w) \\
&\byx{(H8')} m(m(y,m(x,u,z),z),\underbrace{m(m(y,x,z),m(y,x,z),m(y,m(x,u,z),z))},w) \\
&\byx{(P6)} m(m(y,m(x,u,z),z),m(y,m(x,u,z),z),w) \\
&\byx{(H2)} m(w,w,m(y,m(x,u,z),z))\,.
\end{align*}
This completes the proof of the lemma.
\end{proof}

\begin{lemma}
\lemlabel{hard1}
Every algebra $(A,m)$ satisfying (N1) and (N2) satisfies
the identities (H2),
\begin{align}
m(m(x,y,y),y,z) &= m(z,z,y) \eqnlabel{useful1} \\
m(x,y,y) &= y\,. \eqnlabel{useful2}
\end{align}
\end{lemma}

\begin{proof}
First, set $z = y$ in (N2). The left side becomes
\[
m(m(x,y,y),\underbrace{m(y,m(u,x,y),y)},w) \byx{(N1)} m(m(x,y,y),y,w)\,.
\]
The right side becomes
\[
m(w,w,\underbrace{m(y,m(x,u,y),y)}) = m(w,w,y)\,.
\]
The equality of the two sides proves \eqnref{useful1}.

Now (H2) follows easily:
$m(x,x,y) \by{useful1} m(\underbrace{m(y,y,y)},y,x) \byx{(N1)} m(y,y,x)$.

Next, set $z = x$ and $w = m(x,y,x)$ in (N2). The left side becomes
\[
m(m(x,y,x),m(y,m(u,x,x),x),m(x,y,x)) \byx{(N1)} m(x,y,x) \byx{(N1)} x \,.
\]
The right side becomes
\begin{align*}
m(\underbrace{m(x,y,x)},\underbrace{m(x,y,x)},m(y,\underbrace{m(x,u,x)},x))
&\byx{(N1)} m(x,x,m(y,x,x))  \\
&\byx{(H2)} m(m(y,x,x),m(y,x,x),x) \\
&\by{useful1} m(m(y,x,x),x,m(y,x,x)) \\
&\byx{(N1)} m(y,x,x)\,.
\end{align*}
The equality of the two sides proves \eqnref{useful2}.
\end{proof}

\begin{lemma}
For all $x,y,z,u$,
\begin{align}
m(x,m(y,z,u),u) &= m(x,m(z,y,u),u) \eqnlabel{preH4a} \\
m(m(x,y,z),z,u) &= m(u,u,z) \eqnlabel{preH4b} \\
m(x,x,m(y,z,x)) &= m(y,z,x) \eqnlabel{preH4c} \\
m(m(x,y,z),m(y,x,z),z) &= m(x,y,z) \eqnlabel{preH4d} \\
m(m(x,m(y,x,z),z),m(y,x,z),z) &= m(y,x,z) \eqnlabel{preH4e}
\end{align}
\end{lemma}

\begin{proof}
We compute
\begin{align*}
m(x,m(y,z,u),u) &\by{useful2} m(m(z,x,u),m(x,m(y,z,u),u),m(x,m(y,z,u),u)) \\
&\byx{(N2)} m(m(x,m(y,z,u),u),m(x,m(y,z,u),u),m(x,m(z,y,u),u)) \\
&\byx{(H2)} m(m(x,m(z,y,u),u),m(x,m(z,y,u),u),m(x,m(y,z,u),u)) \\
&\byx{(N2)} m(m(y,x,u),m(x,m(z,y,u),u),m(x,m(z,y,u),u)) \\
&\by{useful2} m(x,m(z,y,u),u)\,,
\end{align*}
which establishes \eqnref{preH4a}

Next, in (N2), set $u = z$. The left side becomes
\begin{align*}
m(m(x,y,z),m(y,\underbrace{m(z,x,z)},z),w) &\byx{(N1)} m(m(x,y,z),\underbrace{m(y,z,z)},w) \\
&\by{useful2} m(m(x,y,z),z,w)\,.
\end{align*}
The right side becomes
\begin{align*}
m(w,w,m(y,\underbrace{m(x,z,z)},z)) &\by{useful2} m(w,w,\underbrace{m(y,z,z)}) \\
&\by{useful2} m(w,w,z)\,.
\end{align*}
The equality of the two sides establishes \eqnref{preH4b}.

Now we have
\begin{align*}
m(x,x,m(y,z,x)) &\byx{(H2)} m(m(y,z,x),m(y,z,x),x) \\
&\by{preH4b} m(m(y,z,x),x,m(y,z,x)) \\
&\byx{(N1)} m(y,z,x)\,,
\end{align*}
which gives \eqnref{preH4c}

Next,
\begin{align*}
m(m(x,y,z),m(y,x,z),z) &\by{preH4a} m(m(x,y,z),m(x,y,z),z) \\
&\byx{(H2)} m(z,z,m(x,y,z)) \\
&\by{preH4c} m(x,y,z)\,,
\end{align*}
which yields \eqnref{preH4d}.

Now in (N2), set $y = m(u,x,z)$ and $w = z$. The left side of (N2) becomes
\begin{align*}
&m(m(x,m(u,x,z),z),\underbrace{m(m(u,x,z),m(u,x,z),z)},z) \\
&\byx{(H2)} m(m(x,m(u,x,z),z),\underbrace{m(z,z,m(u,x,z))},z) \\
&\by{preH4c} m(m(x,m(u,x,z),z),m(u,x,z),z)\,.
\end{align*}
The right side becomes
\begin{align*}
m(z,z,\underbrace{m(m(u,x,z),m(x,u,z),z)})
&\by{preH4d} m(z,z,m(u,x,z)) \\
&\by{preH4c} m(u,x,z)\,.
\end{align*}
The equality of the two sides, along with replacing $u$ with $y$, gives
\eqnref{preH4e}.
\end{proof}

\begin{lemma}
\lemlabel{hard2}
Every algebra $(A,m)$ satisfying (N1) and (N2) satisfies (H4).
\end{lemma}

\begin{proof}
In \eqnref{preH4e}, set $x = m(u,v,z)$ and $y = m(v,u,z)$. The left hand side of
\eqnref{preH4e} becomes
\begin{align*}
& m(m(m(u,v,z),\underbrace{m(m(v,u,z),m(u,v,z),z)},z),\underbrace{m(m(v,u,z),m(u,v,z),z)},z) \\
&\by{preH4d} m(\underbrace{m(m(u,v,z),m(v,u,z),z)},m(v,u,z),z) \\
&\by{preH4d} m(m(u,v,z),m(v,u,z),z) \\
&\by{preH4d} m(u,v,z)\,.
\end{align*}
The right side of \eqnref{preH4e} becomes
\begin{align*}
m(m(v,u,z),m(u,v,z),z) \by{preH4d} m(v,u,z)\,.
\end{align*}
The equality of the two sides, along with replacing $u$ with $x$ and $v$ with $y$,
gives (H4).
\end{proof}

\begin{lemma}
\lemlabel{hard3}
Every algebra $(A,m)$ satisfying (N1) and (N2) satisfies (H7).
\end{lemma}

\begin{proof}
We compute
\begin{align*}
m(x,m(x,x,y),z) &\byx{(H4)} m(\underbrace{m(x,x,y)},x,z) \\
&\byx{(H2)} m(m(y,y,x),x,z) \\
&\by{preH4b} m(z,z,x) \\
&\byx{(H2)} m(x,x,z)\,,
\end{align*}
which establishes (H7).
\end{proof}

\begin{lemma}
\lemlabel{hard4}
Every algebra $(A,m)$ satisfying (N1) and (N2) satisfies (H8').
\end{lemma}

\begin{proof}
In (H8'), set $w = m(x,y,z)$. The left side is
\[
m(m(x,y,z),m(y,m(u,x,z),z),m(x,y,z)) \byx{(H2)} m(x,y,z) \byx{(H4)} m(y,x,z)\,.
\]
The right side becomes
\[
m(\underbrace{m(x,y,z)},\underbrace{m(x,y,z)},m(y,m(x,u,z),z)) \byx{(H4)}
m(m(y,x,z),m(y,x,z),m(y,m(x,u,z),z))\,.
\]
The equality of the two sides, along with exchanging the roles of $x$ and $y$,
gives (H8').
\end{proof}

Now we consider the independence of the axioms (N1) and (N2).

\begin{example}
\exmlabel{notN2}
An example of an algebra $(S,m)$ satisfying (N1) but not (N2) is given
by $S =\{0,1\}$ with $m:S\times S\times S\to S$ defined
by the following tables.

\begin{table}[htb]
\centering
\begin{tabular}{r|cc}
$m(0,\cdot,\cdot)$ & 0 & 1 \\
\hline
0 &  0 & 0  \\
1 &  0 & 0
\end{tabular}
\quad
\begin{tabular}{r|cc}
$m(1,\cdot,\cdot)$ & 0 & 1 \\
\hline
0 &  0 & 1 \\
1 &  1 & 1
\end{tabular}
\end{table}
\end{example}

\begin{example}
\exmlabel{notN1}
An example of an algebra $(S,m)$ satisfying (N2) but not (N1) is given
by $S =\{0,1\}$ with $m:S\times S\times S\to S$ defined
by the following tables.

\begin{table}[htb]
\centering
\begin{tabular}{r|cc}
$m(0,\cdot,\cdot)$ & 0 & 1 \\
\hline
0 &  1 & 1  \\
1 &  1 & 1
\end{tabular}
\quad
\begin{tabular}{r|cc}
$m(1,\cdot,\cdot)$ & 0 & 1 \\
\hline
0 &  1 & 1 \\
1 &  1 & 1
\end{tabular}
\end{table}
\end{example}

Putting together Lemmas \lemref{easy}, \lemref{hard1}, \lemref{hard2}, \lemref{hard3}
and \lemref{hard4}, together with Examples \exmref{notN2} and \exmref{notN1},
we have proved Theorem \thmref{2base}.

\section{Problems}
\seclabel{problems}

The following questions arise rather naturally from this investigation.

\begin{problem}
\begin{enumerate}
\item Is there a $2$-base for nearlattices with one axiom no longer than (N1)
and the other shorter than (N2)?
\item Is there a $2$-base for nearlattices involving fewer than five variables?
\item Is the variety of nearlattices $1$-based?
\end{enumerate}
\end{problem}

\begin{acknowledgment}
The investigations in this paper were aided by the automated deduction
tool \textsc{Prover9} and the finite model builder \textsc{Mace4}, both
developed by McCune \cite{McCune}.
\end{acknowledgment}

\end{document}